\documentclass[12pt]{article}

\usepackage[all]{xy}
\SelectTips{eu}{12}

\usepackage{times,eucal,upgreek}

\usepackage{amssymb,amsthm,amsmath}

\numberwithin{equation}{section}

\renewcommand{\AA}{\mathbb{A}}
\renewcommand{\d}{\mathrm{d}}
\renewcommand{\H}{\mathrm{H}}
\renewcommand{\L}{\mathrm{L}}
\renewcommand{\O}{\mathcal{O}}

\renewcommand{\phi}{\varphi}
\renewcommand{\upphi}{\upvarphi}

\newcommand{\uppsican}{{\uppsi_\mathrm{can}}}

\newcommand{\Alg}{\mathrm{Alg}}
\newcommand{\Aut}{\operatorname{Aut}}
\newcommand{\B}{\mathrm{B}}
\newcommand{\Brhat}{\mathrm{\hat Br}}

\newcommand{\calX}{\mathcal{X}}
\newcommand{\calEoo}{\mathcal{E}_\infty}
\newcommand{\D}{\mathrm{D}}
\newcommand{\Eoo}{E_\infty}
\newcommand{\Ext}{\operatorname{Ext}}
\newcommand{\GGmhat}{\mathbb{\hat G}_\mathrm{m}}

\newcommand{\Hcris}{\H_{\mathrm{cris}}}
\newcommand{\HdR}{\H_{\mathrm{dR}}}
\newcommand{\id}{\mathrm{id}}
\newcommand{\K}{\mathrm{K}}
\newcommand{\M}{\mathcal{M}_{\mathrm{K3},p}}
\newcommand{\Mtriv}{\M^\mathrm{triv}}
\newcommand{\Mod}{\mathrm{Mod}}
\newcommand{\Mord}{\M^\mathrm{ord}}
\newcommand{\nablaGM}{\nabla_\mathrm{\!GM}}

\newcommand{\ZZ}{\mathbb{Z}}

\newtheorem{proposition}{Proposition}[subsection]

\newtheorem{theorem}[proposition]{Theorem}
\newtheorem{corollary}[proposition]{Corollary}

\theoremstyle{definition}
\newtheorem{definition}[proposition]{Definition}

\begin{document}

\title{Brave new local moduli\\ for ordinary~K3 surfaces}

\author{Markus Szymik}

\date{August 2009}

\maketitle

\begin{abstract}
\noindent
It is shown that the K3 spectra which refine the local rings of the
moduli stack of ordinary~$p$-primitively polarized K3 surfaces in
characteristic~$p$ allow for an~$\Eoo$ structure which is unique up to
equivalence. This uses the~$\Eoo$ obstruction theory of Goerss and
Hopkins and the description of the deformation theory of such K3
surfaces in terms of their Hodge F-crystals due to Deligne and
Illusie.  Furthermore, all automorphism of such K3 surfaces can be
realized by~$\Eoo$ maps which are unique up to homotopy, and this can
by rigidified to an action if the automorphism group is tame.

\vspace{\baselineskip}	

\end{abstract}

%%%%%%%%%%%%%%%%%%%%%%%%%%%%%%%%%%%%%%%%%%%%%%%%%%%%%%%%%%%%%%%%%%%%%

\section*{Introduction}

In recent years, progress has been made to enrich some classical
moduli stacks of arithmetic origin to objects of stable homotopy
theory, most notably in the case of elliptic curves, see
\cite{hopkins:ICM1994}, \cite{hopkins:ICM2002},
\cite{hopkins+miller:elliptic}, \cite{lurie:elliptic}, but also for
abelian \hbox{varieties}, see \cite{behrens+lawson}, and the Lubin-Tate
moduli of formal groups, see~\cite{hopkins+miller:Lubin-Tate},
\cite{rezk:hopkins+miller}, and~\cite{goerss+hopkins:summary}. To get
an overview, there are the very useful surveys~\cite{goerss:Banff}
and~\cite{goerss:Bourbaki}. This paper pursues that program in the
context of K3 surfaces.

For each odd prime~$p$ there is a Deligne-Mumford moduli stack of
$p$-primitively polarized~K3 surfaces. See~\cite{rizov}, for
example. It is smooth of dimension~19. This stack will be formally
completed at~$p$, and the resulting~$p$-adic moduli stack
of~$p$-primitively polarized~K3 surfaces will be denoted by~$\M$ in
the following. Further decorations will become necessary in due
course. The generic part~$\Mord$ of the moduli stack~$\M$ consists of
the ordinary~K3 surfaces: those whose associated formal Brauer group,
see~\cite{artin+mazur}, is multiplicative. It is this part we will be
dealing with here.

A general idea behind the enrichment of moduli stacks to objects of
stable homotopy theory is to replace the structure sheaves of rings on
the moduli stacks by sheaves of ring spectra, objects which represent
multiplicative cohomology theories. Ring spectra nowadays come in two
kinds of precision: the older `up to homotopy' versions, and the more
recent `brave new rings' version. See~\cite{may+quinn+ray+tornehave}
for the classic text on the latter,
and~\cite{mandell+may+schwede+shipley} for a comparison of many of the
more recent models. A {\it K3 spectrum} is a triple~$(E,X,\iota)$,
where~$E$ is an even periodic ring spectrum `up to homotopy',~$X$ is a
K3 surface over~$\pi_0E$, and~$\iota$ is an isomorphism of the formal
Brauer group~$\Brhat_X$ of~$X$ with the formal group associated with
such an~$E$, see~\cite{szymik:K3spectra}, where it is also proven that
all local rings of~$\M$ at K3 surfaces~$X$ of finite height and their
formal completions can be realized as underlying rings~$\pi_0E$ for
suitable~K3 spectra~$(E,X,\iota)$.

The aim of this writing is to enhance the multiplications on these K3
spectra from good old `up to homotopy' to brave new `highly
structured' versions in the ordinary case. Although this is still less
than having a sheaf of~$\Eoo$ ring spectra on the ordinary locus, this
already tells us what the stalks will be.  I will return
to the construction of a sheaf of~$\Eoo$ ring spectra on the ordinary
locus (and beyond) somewhere else.

There is a good reason why the local question in the K3 case should be
thought of as the essential step: In contrast to elliptic curves and
abelian varieties, where the local deformation theory is reduced to
the deformation theory of the associated Barsotti-Tate groups by means
of the Serre-Tate theorem, see~\cite{lubin+serre+tate},
\cite{messing:serre-tate}, \cite{drinfeld}, \cite{katz:serre-tate},
and~\cite{illusie:barsotti-tate}, this does not hold for polarized K3
surfaces in general, and not obviously so in the ordinary case,
although~\cite{nygaard:Tate_for_ordinary} proves a result along these
lines for K3 surfaces without polarizations. Instead, it seems that K3
surfaces will have to be dealt with by means of their crystalline
invariants, and this is the optic in which they will be viewed
here. The formal Brauer group associated with a K3 surface will
sometimes be mentioned for the benefit of the traditionalists, but the
the mindset of algebraic topologists slowly seemed to change from
formal group laws over formal groups to Barsotti-Tate groups. The next
step towards crystals now seems inevitable, and this may well be
considered as the primary novelty introduced here.

Here is an outline of the following text: In
Sections~\ref{sec:ordinary} and~\ref{sec:trivialized}, we will discuss
some structure present on the local moduli spaces of ordinary and
trivialized~K3 surfaces with polarization, respectively. As it turns
out, this is exactly the structure observed on the~$p$-adic
$\K$-homology of~$\K(1)$-local~$\Eoo$ ring spectra. Thus, it may serve
as an input for the obstruction theory of Goerss and Hopkins, which is
reviewed in Section~\ref{sec:GoerssHopkins}. In
Section~\ref{sec:applications}, this will be applied to prove the
existence and uniqueness of an~$\Eoo$ structure on said ring
spectra. The final Section~\ref{sec:symmetries} exploits the
symmetries of ordinary K3 surfaces, in other words, the stackiness of
$\M$.

I would like to thank Mike Hopkins and Haynes Miller for their
generous hospitality in different stages of this project, Mark Behrens
and Jacob Lurie for explanations of various topics related to it, and
Thomas Zink for listening to and commenting on my crystalline
problems.

%%%%%%%%%%%%%%%%%%%%%%%%%%%%%%%%%%%%%%%%%%%%%%%%%%%%%%%%%%%%%%%%%%%%%

\section{Ordinary K3 surfaces}\label{sec:ordinary}

In this section, we will review the local deformation theory of
polarized K3 surfaces, especially of the ordinary ones, from the point
of view of their crystals. The main references are
\cite{ogus:supersingular}, \cite{deligne+illusie:relevements},
\cite{deligne+illusie:cristaux}, and \cite{katz:appendix}.

%%%%%%%%%%%%%%%%%%%%%%%%%%%%%%%%%%%%%%%%%%%%%%%%%%%%%%%%%%%%%%%%%%%%%

\subsection{K3 surfaces and their deformations}

Let~$k$ be an algebraically closed field. A {\it K3
  surface}~\hbox{$X$} over~$k$ is a smooth projective surface over~$k$
such that its canonical bundle~$\Omega^2_{X/k}$ is trivial and such
that~$X$ is not abelian. Examples are the {\it Fermat quartic} defined
by~\hbox{$T^4_1+T^4_2+T^4_3+T^4_4$} in~$\mathbb{P}^3_k$, more
generally any smooth hypersurface of degree~$4$, and the {\it Kummer
  surfaces}, which are obtained by extending the inversion on an
abelian surface over the blowup at the 16 fixed points and passing to
the quotient. In these examples and from now on it will be assumed
that~$k$ is of odd characteristic~$p$.

The Hodge diamond of a K3 surface~$X$, which symmetrically displays
its Hodge numbers~$\dim_k\H^j(X,\Omega^i_{X/k})$, looks as follows.
\begin{displaymath}
\begin{array}{ccccc}
&&1&&\\
&0&&0&\\
1&&20&&1\\
&0&&0&\\
&&1&&
\end{array}
\end{displaymath}
This implies that the Hodge-to-de Rham spectral sequence has to
degenerate at~$E_1$. In particular, there are no obstructions to
extending deformations, the tangent space to the deformation functor
has dimension~$20$, and there are no infinitesimal automorphisms. This
gives the following result, where~$W$ denotes the ring of~$p$-typical
Witt vectors of~$k$.

\begin{theorem}{\upshape(\cite{deligne+illusie:relevements}, 1.2)}
  The formal deformation space~$S$ of~$X$ over~$W$ is formally smooth
  of dimension~$20$, so that there is a non-canonical isomorphism
  \begin{displaymath}
    S\cong\AA^{\!20}_W\,,
  \end{displaymath}
  and there is a universal formal deformation~$\calX$ over~$S$.
\end{theorem}

Further down, see Theorem~\ref{thm:cc}, we will see that there is a
particularly useful set of coordinates for~$S$ in the ordinary case.

A {\it polarized K3 surface} is a pair~$(X,L)$, where~$X$ is a~K3
surface and~$L$ is an ample line bundle on~$X$. We shall always assume
that the polarization is~{\it~$p$-primitive} for the chosen prime~$p$
in the sense that~$L$ is not isomorphic to the~$p$-th power of another
line bundle. This implies that~$p$ does not divide the degree~$L^2$
of~$L$.

\parbox{\linewidth}{\begin{theorem}{\upshape(\cite{deligne+illusie:relevements}, 1.5 and 1.6)} Let~$L$ be a polarization on~$X$ as above. The formal
  deformation space of~$(X,L)$ is representable by a closed formal
  subscheme~$S_L\subset S$, defined by a single equation. It is flat
  over~$W$ of relative dimension~$19$.
\end{theorem}}

Note that flatness implies that~$p$ does not divide an equation
defining~$S_L$. Further down, see Theorem~\ref{thm:c1}, a more precise
formula for such an equation will be given in the ordinary case, and
this will show that~$S_L$ is in fact formally smooth. 

%%%%%%%%%%%%%%%%%%%%%%%%%%%%%%%%%%%%%%%%%%%%%%%%%%%%%%%%%%%%%%%%%%%%%

\subsection{Crystals associated with K3 surfaces}\label{sec:K3crystals}

Here it will be explained, following~\cite{deligne+illusie:cristaux},
2.2, how to associate a Hodge F-crystal to a K3 surface, and what it
means for that crystal, and therefore by definition for the K3
surface, to be ordinary.

As before, let~$\calX$ be a universal formal deformation over~$S$ of a
K3 surface~$X$. Then the~$\O(S)$-module
\begin{displaymath}
	H=\HdR^2(\calX/S),
\end{displaymath}
together with the Gauss-Manin connection~$\nabla=\nablaGM$ is a
crystal.

If~$\upphi$ is a lift of Frobenius to~$S$ which is compatible with
the canonical lift of Frobenius to~$W$, there is an induced~$\upphi$-linear map
\begin{displaymath}
	F_\upphi\colon H\longrightarrow H.
\end{displaymath}
This would follow immediately from the existence of
an~$S$-morphism~$\calX\rightarrow\upphi^*\calX$ which lifts the
relative Frobenius of~$X$, as~$F_\upphi$ could be defined as the
composition
\begin{displaymath}
  \upphi^*\HdR^2(\calX/S)\cong\HdR^2(\upphi^*\calX/S)\longrightarrow\HdR^2(\calX/S).
\end{displaymath}
However, such an arrow need not exist. But its mod~$p$ reduction, the
relative Frobenius, always exists. Thus, one may use (a) the canonical
isomorphism between the de Rham cohomology of~$\calX$ and the
crystalline cohomology of its reduction, and (b) the functoriality of
crystalline cohomology to obtain the desired maps. Summing up, this
means that~$(H,\nabla,F_\bullet)$ is an F-crystal. Note that some such
lift~$\upphi$ of Frobenius always exists by the formal smoothness
of~$S$. Later on, see Section~\ref{sec:Katz_lift}, a particular lift
will be distinguished in the ordinary case.

The Hodge filtration
\begin{displaymath}
	H=F^0\supset F^1\supset F^2\supset F^3=0
\end{displaymath}
lifts the Hodge filtration on the reduction and satisfies the
so-called Griffiths transversality condition. In other words,
$(H,\nabla,F_\bullet,F^\bullet)$ is a Hodge F-crystal.

A K3 surface is {\it ordinary} if the Hodge and Newton polygons of its
associated Hodge F-crystal agree, see~\cite{mazur:bulletin},
Section~3. The Hodge polygon codifies the Hodge numbers, as is
illustrated in the following figure.

\vbox{\small
  \begin{center} 
    \unitlength0.5cm
    \begin{picture}(6,6)(0,0)
    \thicklines

      \multiput(0,0)(.15,.15){40}{\makebox(0,0){$\cdot$}}
      \multiput(0,0)(0,.2){30}{\makebox(0,0){$\cdot$}}
      \multiput(1,0)(.2,0){25}{\makebox(0,0){$\cdot$}}
      \multiput(6,0)(0,.2){30}{\makebox(0,0){$\cdot$}}
      \multiput(0,6)(.2,0){30}{\makebox(0,0){$\cdot$}}

      \put(-1,0){\makebox(0,0){$0$}}
      \put(-1,4){\makebox(0,0){$20$}}
      \put(-1,6){\makebox(0,0){$22$}}
    
      \put(0,-1){\makebox(0,0){$0$}}    
      \put(0,0){\makebox(0,0){$\bullet$}}
      
      \put(0,0){\line(1,0){1}}
      
      \put(1,-1){\makebox(0,0){$1$}}
      \put(1,0){\makebox(0,0){$\bullet$}}
      
      \put(1,0){\line(1,1){4}}

      \put(5,-1){\makebox(0,0){$21$}}
      \put(5,4){\makebox(0,0){$\bullet$}}
      
      \put(5,4){\line(1,2){1}}

      \put(6,-1){\makebox(0,0){$22$}}
      \put(6,6){\makebox(0,0){$\bullet$}}
    \end{picture}

	\vspace{\baselineskip}

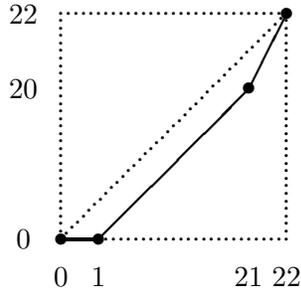
	Figure: The Newton/Hodge polygon of an ordinary K3 surface
        (not drawn to scale)
\end{center}}

The Newton polygon codifies the multiplicities and~$p$-adic valuations
of the eigenvalues of Frobenius. In the case of a K3 surface, this
suggests correctly that the first slope of the Newton polygon is zero
if and only if Frobenius acts on~$\H^2(X,\O_X)$ bijectively. It is
also known that the first slope is~$1-(1/h)$, where~$h$ is the height
of the formal Brauer group. See~\cite{illusie:deRham-Witt}, 7.2, for
example. Thus, a K3 surface is ordinary if and only if its formal
Brauer group is multiplicative.

An ordinary Hodge F-crystal~$(H,\nabla,F_\bullet,F^\bullet)$ of level
$n$, where~$n=2$ for K3 surfaces, has a filtration
\begin{displaymath}
	0\subset U_0\subset U_1\subset\cdots\subset U_n=H
\end{displaymath}
such that Frobenius acts on~$U_j/U_{j-1}$ as the~$p^j$-th multiple
of a bijection, and this filtration is opposite to the Hodge
filtration in the sense that
\begin{displaymath}
	H=\bigoplus_j\,(U_j\cap F^j).
\end{displaymath}
This again characterizes ordinary Hodge F-crystals,
see~\cite{deligne+illusie:cristaux}, 1.3.2.

%%%%%%%%%%%%%%%%%%%%%%%%%%%%%%%%%%%%%%%%%%%%%%%%%%%%%%%%%%%%%%%%%%%%%

\subsection{Canonical coordinates and the Katz lift}\label{sec:Katz_lift}

The associated Hodge F-crystal of an ordinary K3 surface, as described
in Section~\ref{sec:K3crystals}, has a particularly nice structure,
and this can be used to find particularly nice coordinates on the base
$S$ of its universal formal deformation.

\begin{theorem}
	{\upshape (\cite{deligne+illusie:cristaux}, 2.1.7)}\label{thm:cc}
	\label{thm:DI:2.1.7}
	Let~$X$ be an ordinary K3 surface with universal
        formal deformation~$\calX$ over~$S$. Then there is a
        basis~$(a,b_1,\dots,b_{20},c)$ for the associated crystal as
        well as coordinates~$t_1,\dots,t_{20}$ on~$S$ such that the
        following properties {\upshape (\ref{thm:DI:2.1.7}.1) --
          (\ref{thm:DI:2.1.7}.4)} hold.
\end{theorem}
	
{\upshape (\ref{thm:DI:2.1.7}.1)} The basis is adapted to the
decomposition
\begin{displaymath}
	H=U_0\oplus(U_1\cap F^1)\oplus F^2
\end{displaymath}
and satisfies~$\langle a,b_j\rangle=0$,~$\langle
b_j,c\rangle=0$,~$\langle a,a\rangle=0$,~$\langle c,c\rangle=0$,
and~$\langle a,c\rangle=1$, where~$\langle\cdot,\cdot\rangle$ denotes
the cup-product pairing on middle cohomology.
	
{\upshape (\ref{thm:DI:2.1.7}.2)} If use multiplicative
notation~$q_j=t_j+1$ and~$\omega_j=\d\!\log(q_j)$, then~$(\omega_j)$
is a~$W$-basis of~$\Omega_{S/W}$.
	
{\upshape (\ref{thm:DI:2.1.7}.3)} The Gauss-Manin connection acts via
\begin{displaymath}
	\nablaGM(a)=0\,,\quad
	\nablaGM(b_j)=\omega_j\otimes a\,,\quad
	\nablaGM(c)=-\sum_j\omega_j\otimes b_j^\vee\,,
\end{displaymath}
where~$(b_j^\vee)$ is the cup-dual basis to~$(b_j)$.

{\upshape (\ref{thm:DI:2.1.7}.3)} If~$\uppsican$ is the lift of
Frobenius given by~$\uppsican(q_j)=q_j^p$, then the
induced~$\uppsican$-linear map~$F_\uppsican$ on~$H=\HdR^2(\calX/S)$ is
given by
\begin{displaymath}
		F_\uppsican(a)=a\,,\quad
	F_\uppsican(b_j)=pb_i\,,\quad
	F_\uppsican(c)=p^2c\,.
\end{displaymath}

A system~$(a,b,c,t)$ as in the preceding theorem is called a {\it
  system of canonical coordinates} on~$S$, and~$\uppsican$ will be
referred to as the {\it canonical lift} or {\it Katz
  lift}~(\hbox{after}~\cite{katz:appendix}) of Frobenius. The term
{\it Deligne-Tate mapping} is also in use.
  
While a system of canonical coordinates is not unique, there is only a
rather restricted choice involved: If~\hbox{$(a',b',c',t')$} is
another system, there are~$\alpha\in\ZZ_p^\times$
and~\hbox{$\beta=(\beta_{ij})\in\operatorname{GL}_{20}(\ZZ_p)$} such
that
\begin{displaymath}
	a'=\alpha a,\quad
	b'_i=\sum_j\beta_{ji}b_j,\quad
	c'=c/\alpha,
\end{displaymath}
and
\begin{displaymath}
	q'_i=\prod_jq_j^{\beta_{ji}/\alpha}.
\end{displaymath}
See~\cite{deligne+illusie:cristaux}, 2.1.13. In particular, this shows
that the Katz lift does not depend on the canonical coordinates. It is
intrinsic to the situation. As the notation~\hbox{$q_j=t_j+1$}
and~$\omega_j=\d\!\log(q_j)$ already indicates, these coordinates can
be used to identify~$S$ with a formal torus
\begin{displaymath}
	S\cong\GGmhat^{20},
\end{displaymath}
as in~\cite{deligne+illusie:cristaux}. (A description of this group
structure on~$S$ without the use of the canonical coordinates has been
given in~\cite{nygaard:Tate_for_ordinary}.) The Katz lift is a formal
group homomorphism and the unique lift of Frobenius whose
associated~$F$ preserves the Hodge filtration,
see~\cite{katz:appendix}, A4.1.

If~$X$ is ordinary, and~$L$ is~$p$-primitive, the base~$S_L$ of the
universal formal deformation of~$(X,L)$ is not only flat but even
formally smooth, see~\cite{ogus:supersingular}, 2.2. More can be said
using a system~$(a,b,c,t)$ of canonical coordinates as above,
see~\cite{deligne+illusie:cristaux}, 2.2. Let
\begin{displaymath}
	e\colon\O(S)\longrightarrow W
\end{displaymath}
be the co-unit given by~$q_j\mapsto1$. Then the first crystalline
Chern class of~$L$ can be written
\begin{displaymath}
	\sum_j\lambda_je^*b_j
\end{displaymath}
with~$p$-adic integers~$\lambda_j$. As the first crystalline Chern
class of a~$p$-primitive line bundle is not divisible by~$p$, some
$\lambda_j$ will in fact be a~$p$-adic unit.

\begin{theorem}{\upshape(\cite{deligne+illusie:cristaux}, 2.2.2)}\label{thm:c1}
  In the notation as before,
\begin{displaymath}
		\prod_{j=1}^{20}q_j^{\lambda_j}=1
\end{displaymath}
	is an equation for~$S_L$ in~$S$.
\end{theorem}

In other words, we can interpret the first crystalline Chern class as
a character of the formal torus~$S$, and~$S_L$ is its kernel.

\begin{proposition}
	The Katz lift~$\uppsican$ on~$S$ maps~$S_L$ into itself.
\end{proposition}

\begin{proof}
  As~$S_L$ is defined in~$S$ by~$\prod_jq_j^{\lambda_j}=1$, the
  computation
	\begin{displaymath}
          \uppsican\big(\prod_jq_j^{\lambda_j}-1\big)
          =\prod_j\uppsican(q_j)^{\lambda_j}-1
          =\prod_j(q_j^p)^{\lambda_j}-1=\big(\prod_jq_j^{\lambda_j}\big)^p-1=0
	\end{displaymath}	
	shows that~$\uppsican$ preserves the equation.
\end{proof}

%%%%%%%%%%%%%%%%%%%%%%%%%%%%%%%%%%%%%%%%%%%%%%%%%%%%%%%%%%%%%%%%%%%%%

\section{Trivialized K3 surfaces}\label{sec:trivialized}

In the section, the analogue for K3 surfaces of Katz' notion of
trivialized elliptic curves will be discussed. See~\cite{katz:arcata},
\cite{katz:higher}, \cite{katz:eisenstein} for the latter.

\subsection{The rigidified moduli stack~$\Mtriv$}\label{sec:Mtriv}

Recall that~$\M$ is the~$p$-adic moduli stack of~$p$-primitively
polarized~K3 surfaces. Now let~$\Mord$ denote the open substack of
$\M$ consisting of the ordinary surfaces. 

\begin{definition}
	The rigidified moduli stack
	\begin{displaymath}
		\Mtriv
	\end{displaymath}
	classifies {\it trivialized K3 surfaces}: triples~$(X,L,a)$ of
        ordinary K3 surfaces~$X$ together with a~$p$-primitive
        polarization~$L$, and an element~$a$ of~$H$ which is
        the~$a$-part of a system of canonical coordinates: it is
        annihilated by the Gauss-Manin connection and left invariant
        by the Katz lift of Frobenius.
\end{definition}

The choice of~$a$ corresponds to the choice of an
isomorphism~$\Brhat_{X}\cong\GGmhat$, and it is in this way how Katz
introduced trivializations. However, for our purposes, the definition
given above corresponds to the crystalline mindset taken here, and it
turns out to be easier to work with, too.

It follows from the discussion in the previous
Section~\ref{sec:Katz_lift} that there is a free and transitive action
of the group~$\ZZ_p^\times$ of the~$p$-adic units on the fibers of the
forgetful morphism
\begin{displaymath}
  \Mtriv\longrightarrow\Mord.
\end{displaymath}
This yields a Galois covering with group~$\ZZ_p^\times$.

Let us now fix a~$p$-primitively polarized K3 surface~$(X,L)$, and
let~$S_L$ be the base of a universal formal deformation of
it. As~$\Mtriv$ is Galois over~$\Mord$, so is the pullback~$T_L$ along
the classifying morphism for the universal family.
\begin{center}
  \mbox{ \xymatrix{
      T_L\ar[d]\ar[r]& \Mtriv\ar[d] \\
      S_L\ar[r] & \Mord } }
\end{center}
As~$S_L$ is affine, so is~$T_L$. We are now going to see some
structure on its ring~$\O(T_L)$ of formal functions.

%%%%%%%%%%%%%%%%%%%%%%%%%%%%%%%%%%%%%%%%%%%%%%%%%%%%%%%%%%%%%%%%%%%%%

\subsection{The Adams operations}\label{sec:Adams}

Let~$T_L$ be as in the end of the previous
Section~\ref{sec:Mtriv}. The Galois action
of~$\Aut(\GGmhat)\cong\ZZ_p^\times$ on~$\Mtriv$ restricts to~$T_L$,
and the corresponding operations on the ring~$\O(T_L)$ of formal
functions will be denoted by~$\uppsi^k$~for~$p$-adic units~$k$. These
will be referred to as the {\it Adams operations}, a terminology which
will be justified in the following Section~\ref{sec:GoerssHopkins}.

Let us denote by~$\omega$ the Hodge line bundle over~$\M$ whose fiber
over~$(X,L)$ is the line~$\H^0(X,\Omega^2_{X})$ of regular 2-forms on
$X$. The same notation will be used for its restriction to~$\Mord$,
$\Mtriv$,~$S_L$, and~$T_L$ as needed. As the operators~$\uppsi^k$
on~$\O(T_L)$ are induced by an action on~$T_L$, it is clear, that we
have a similar action on~\hbox{$\H^0(T_L,\omega^{\otimes n})$} for
each integer~$n$.

%%%%%%%%%%%%%%%%%%%%%%%%%%%%%%%%%%%%%%%%%%%%%%%%%%%%%%%%%%%%%%%%%%%%%

\subsection{The operator~$\uptheta$}\label{sec:theta}

In addition to the lift of the action of the~$p$-adic units to~$T_L$,
we will now explain how the Katz lift~$\uppsican$ on~$\O(S_L)$ can be
extended to~$\O(T_L)$ as well. To do so, we need to produce, from the
given (universal)~$a$ over~$T_L$, another such element
for~$\uppsi_\mathrm{can}^*\calX$ in place of~$\calX$. This is easy to
do from the crystalline point of view on trivializations: The
morphism~$\uppsi_\mathrm{can}^*\calX\rightarrow\calX$ induces a
morphism
\begin{displaymath}
	\uppsi_\mathrm{can}^*\colon
	\HdR^2(\calX/S_L)
	\longrightarrow
	\HdR^2(\uppsi_\mathrm{can}^*\calX/S_L)
\end{displaymath}
which sends~$a$ to some element~$\uppsi_\mathrm{can}^*a$. The
resulting selfmap of~$T_L$ will be denoted by~$\uppsican$ as well.

\begin{proposition}
	The Katz lift~$\uppsican$ of Frobenius
	determines a unique operation~$\uptheta$ on~$\O(T_L)$ such that
	\begin{displaymath}
		\uppsican(f)=f^p+p\uptheta(f)
	\end{displaymath} 
	holds for each~$f$ in~$\O(T_L)$.
\end{proposition}

\begin{proof}
  As~$\uppsican$ is a lift of Frobenius, there is always one
  such~$\uptheta(f)$ which satisfies the equation. This shows
  that~$\uptheta$ exists.

  As~$T_L$ is flat (even formally smooth) over~$W$, multiplication
  by~$p$ is injective on~$\O(T_L)$. This shows that~$\uptheta$ is
  unique.
\end{proof}

Using terminology explained in the following
Section~\ref{sec:theta_algebras}, the same argument provides for the
structure of a graded~$\uptheta$-algebra with Adams operations on
\begin{displaymath}
  \left(\,\H^0(T_L,\omega^{\otimes n})\,|\,n\in\ZZ\,\right).
\end{displaymath}
This object can and will serve as a blueprint from which the~$\Eoo$
structures on the local K3 spectra mentioned in the introduction can
be (re)constructed, using the obstruction theory described in the
following section.

%%%%%%%%%%%%%%%%%%%%%%%%%%%%%%%%%%%%%%%%%%%%%%%%%%%%%%%%%%%%%%%%%%%%%

\section{Goerss-Hopkins obstruction theory}\label{sec:GoerssHopkins}

In this section, we review the work of Goerss and Hopkins
on~$\K(1)$-local~$\Eoo$ ring spectra and spaces of~$\Eoo$ maps between
them. An odd prime~$p$ is fixed throughout. References
are~\cite{hopkins:K(1)},~\cite{goerss+hopkins:andre},
\cite{goerss+hopkins:summary}, and~\cite{goerss+hopkins:obstructions}.

%%%%%%%%%%%%%%%%%%%%%%%%%%%%%%%%%%%%%%%%%%%%%%%%%%%%%%%%%%%%%%%%%%%%%

\subsection{The theory of~$\uptheta$-algebras}\label{sec:theta_algebras}

Let~$\K$ denote the~$p$-adic completion of the topological
complex~K-theory spectrum. In a broader context, this is also known as
the first Lubin-Tate spectrum~$E_1=E(\mathbb{F}_p,\GGmhat)$. It has
an~$\Eoo$ structure such that the (stable) Adams
operations~\hbox{$\uppsi^k\colon\K\rightarrow\K$}~(for~$p$-adic
units~$k$) are~$\Eoo$ maps. Therefore, if~$X$ is any spectrum,
the~$\K$-homology~\hbox{$\K_0X=\pi_0(\K\wedge X)$} also has these
operations. As everything has to be~$\K(1)$-local, smash products such
as~\hbox{$\K\wedge X$} will implicitly be~$\K(1)$-localized.

If~$E$ is a~$\K(1)$-local~$\Eoo$ ring spectrum, the underlying
ring~$\pi_0E$ is a so-called~{\it~$\uptheta$-algebra}. This means that
there are two operations~$\uppsi^p$ and~$\uptheta$ on~$\pi_0E$ which come
about as follows. Given a class~$x$ in~$\pi_0E$, the~$\Eoo$ structure on~$E$
produces a morphism
\begin{displaymath}
	P(x)\colon\B\Sigma_{p+}\longrightarrow E
\end{displaymath}
which restricts to~$x^p$ along the inclusion~$e\colon
\mathrm{S}^0=\B1_{+}\rightarrow\B\Sigma_{p+}$. In the~$\K(1)$-local category
there are two other distinguished morphisms
\begin{displaymath}
	\uppsi^p,\uptheta\colon \mathrm{S}^0\rightarrow\B\Sigma_{p+},
\end{displaymath}
and the `restriction' of~$P(x)$ along these will be denoted
by~$\uppsi^p(x)$ and~$\uptheta(x)$. For example, if~$X$ is a space, the
function spectrum~$\K^X$ is a~$\K(1)$-local~$\Eoo$ ring spectrum
with~$\pi_0(\K^X)=\K^0(X)$, and~$\uppsi^p$ is the~$p$-th (unstable) Adams
operation, whereas~$\uptheta$ is Atiyah's operation~\cite{atiyah}.
In general, the equation~$e=\uppsi^p-p\uptheta$ implies the relation
\begin{displaymath}
	\uppsi^p(x)=x^p+p\uptheta(x)
\end{displaymath}
for all~$x$ in~$\pi_0E$ so that~$\uppsi^p$ is a lift of Frobenius
on~$(\pi_0E)/p$ and~$\uptheta$ is the error term. This also means that
the operation~$\uptheta$ determines the operation~$\uppsi^p$. The converse
holds if the ring is~$p$-torsion free.

While the operation~$\uppsi^p$ is a ring map, the map~$\uptheta$ satisfies the
following equations.
\begin{eqnarray*}
	\uptheta(x+y)&=&
        \uptheta(x)+\uptheta(y)-\sum_{j=1}^{p-1}\binom{p}{j}x^jy^{p-j}\\
	\uptheta(x\cdot y)&=&x^p\uptheta(y)+y^p\uptheta(x)+
	p\uptheta(x)\uptheta(y)\phantom{\binom{p}{j}}
\end{eqnarray*}
These are best phrased saying that~$s=(\id,\uptheta)$ is a ring map to the
ring of Witt vectors of length 2 which defines a section of the first
Witt component~$w_0$. As the other Witt component is given by
\begin{displaymath}
w_1(a_0,a_1)=a_0^p+pa_1, 
\end{displaymath}
composition of the latter with~$s=(\id,\uptheta)$ then gives
$\uppsi^p$.\vspace{-10pt}
\begin{center}
  \mbox{ 
    \xymatrix@C=40pt@R=40pt{& 
    A\ar@{=}[dl]_{}\ar[dr]^{\uppsi^p}\ar[d]|-{(\id,\,\uptheta)} & \\
		A & W_2A\ar[l]^-{w_0}\ar[r]_-{w_1} & A
    } 
  }
\end{center}

Putting the two structures together, if~$E$ is a~$\K(1)$-local~$\Eoo$
ring spectrum, the underlying ring~$\K_0E=\pi_0(K\wedge E)$ is
a~$\uptheta$-algebra with Adams operations. This is the primary algebraic
invariant of the~$\K(1)$-local~$\Eoo$ ring spectrum~$E$, and the obstruction
theory laid out in the following describes how good this invariant is.

There is a graded version of the previous notions which is modeled to
capture the structure present on~$\K_*E$ rather than just~$\K_0E$, see
Definition 2.2.3 in~\cite{goerss+hopkins:obstructions}. In the case at
hand, where we are dealing with even periodic~$E$, these
contain essentially the same information as their degree zero part, and we
will not go into detail here.

%%%%%%%%%%%%%%%%%%%%%%%%%%%%%%%%%%%%%%%%%%%%%%%%%%%%%%%%%%%%%%%%%%%%%

\subsection{Existence and uniqueness of~$\Eoo$ structures}

Goerss and Hopkins address the following question: Given a
graded~$\uptheta$-algebra~$B_*$ with Adams operations, when is there
a~$\K(1)$-local~$\Eoo$ ring spectrum~$E$ such that
\begin{equation}\label{eq:K_*}
	\K_*E\cong B_*
\end{equation}
as~$\uptheta$-algebras with Adams operations? Their answer is as follows.

\parbox{\linewidth}{\begin{proposition}\label{prop:structures}
  {\upshape(\cite{goerss+hopkins:summary}, 5.9, and~\cite{goerss+hopkins:obstructions}, 3.3.7)} Given a
  graded~$\uptheta$-algebra~$B_*$ with Adams operations, there exists
  a~$\K(1)$-local~$\Eoo$ ring spectrum~$E$ such that~$\K_*E\cong B_*$
  as~$\uptheta$-algebras with Adams operations if certain obstruction groups
  \begin{displaymath}
    \D^{t+2,t}_{\uptheta\Alg/\K_*}(B_*/\K_*,B_*)
  \end{displaymath}
  vanish for all~$t\geqslant1$. Furthermore, the~$\Eoo$ structure is
  unique up to equivalence if the groups
  \begin{displaymath}
    \D^{t+1,t}_{\uptheta\Alg/\K_*}(B_*/\K_*,B_*)
  \end{displaymath}
  vanish for all~$t\geqslant1$.
\end{proposition}}

Uniqueness here does not mean that there are no non-trivial
automorphisms; in fact there usually will be many. It only says that
two such objects will be equivalent, in possibly many different ways.

The theorem can be thought of as the obstruction theory for a spectral
sequence with~$E_2$ term
\begin{displaymath}
	E_2^{s,t}=\D^{s,t}_{\uptheta\Alg/\K_*}(B_*/\K_*,B_*),
\end{displaymath}
trying to converge to the homotopy groups~$\pi_{t-s}$ of an
appropriate space of all such realizations.

Rather than defining the obstruction groups, we will only describe --
in Subsection~\ref{subsec:techniques} -- methods to compute them, as
this will be what is needed for the applications. It should be
mentioned, however, that the letter~$\D$ stands for~`derivations', and
the obstruction groups come about as a topological version of the
Andr\'e-Quillen theory of non-abelian derived derivations. However,
see the next subsection for a hint why derivations come in. 

The coefficients~$M_*$ of the obstruction
groups~$\D^s_{\uptheta\Alg/K_*}(B_*/K_*,M_*)$ for
a~$\uptheta$-algebra~$B_*$ are~{\it~$\uptheta$-modules} in the sense
of Definition~2.2.7 in~\cite{goerss+hopkins:obstructions}. These
are~$B_*$-modules with the structure of a~$\uptheta$-algebra
on~$B_*\oplus M_*$, which is essentially given by a map~$\uptheta$
on~$M_*$ that satisfies~$\uptheta(bm)=\uppsi(b)\uptheta(m)$ if both
have even degree.

%%%%%%%%%%%%%%%%%%%%%%%%%%%%%%%%%%%%%%%%%%%%%%%%%%%%%%%%%%%%%%%%%%%%%

\subsection{Spaces of~$\Eoo$ maps}

We will also have occasion to employ the obstruction theory for spaces
of~$\Eoo$ maps between~$\K(1)$-local~$\Eoo$ spectra~$E$ and~$F$. In
fact, this may be easier to grasp than the obstruction theory
for~$\Eoo$ structures, which can be thought of as a theory to realize
the identity as an~$\Eoo$ map. In particular, the obstruction groups
will be the same as before, so that the same computational methods
will apply. The reference for the material here
is~\cite{goerss+hopkins:summary}, Section~4,
and~\cite{goerss+hopkins:obstructions}, Section~2.4.4.

\parbox{\linewidth}{%
  \begin{proposition}\label{prop:maps-existence}
    {\upshape(\cite{goerss+hopkins:summary}, 4.4, \cite{goerss+hopkins:obstructions}, 2.4.15)}
    Let~$E$ and~$F$ be~$\K(1)$-local~$\Eoo$ ring spectra, and
    let~\hbox{$d_*\colon \K_*E\rightarrow \K_*F$} be a map
    of~$\uptheta$-algebras over~$\K_*$. The obstructions to
    realizing~$d_*$ as the~$\K$-homology of an~$\Eoo$
    map~\hbox{$g\colon E\rightarrow F$} lie in groups
  \begin{displaymath}
    \D^{t+1,t}_{\uptheta\Alg/\K_*}(\K_*E/\K_*,\K_*F)
  \end{displaymath}
  for~$t\geqslant1$.
\end{proposition}}

Assuming that such a~$g$ exists, the obstructions for uniqueness lie
in groups
\begin{displaymath}
  \D^{t,t}_{\uptheta\Alg/\K_*}(\K_*E/\K_*,\K_*F)
\end{displaymath}
for~$t\geqslant1$. 

Again this is just the beginning of a spectral sequence which computes
the homotopy groups of the component~$\calEoo(E,F)_g$ of~$g$ in the
space of~$\Eoo$ maps. The idea behind the construction of this
spectral sequence is to use the cosimplicial resolution of the
source~$E$ by the triple of the standard adjunction between~$\Eoo$
ring spectra and~$\K$-algebras. The precise statement is as follows.

\parbox{\linewidth}{\begin{proposition}
	\label{prop:maps-uniqueness}
	{\upshape(\cite{goerss+hopkins:summary}, 4.3, \cite{goerss+hopkins:obstructions}, 2.4.14)}
  Given an~$\Eoo$ map~$g\colon E\rightarrow F$ of~$\K(1)$-local~$\Eoo$
  ring spectra, there is a spectral sequence
  \begin{displaymath}
    \D^{s,t}_{\uptheta\Alg/\K_*}(\K_*E/\K_*,\K_*F)
    \Longrightarrow\pi_{t-s}	\calEoo(E,F)_g
  \end{displaymath}
  converging to the homotopy groups of the component of~$g$ in the
  space of~$\Eoo$ maps from~$E$ to~$F$.
\end{proposition}}

It it easy to see where derivations come into the picture here. If
$g\colon E\rightarrow F$ is an~$\Eoo$ map, and
\begin{displaymath}
	\mathrm{S}^n\longrightarrow\mathcal{E}_\infty(E,F)
\end{displaymath}
is a map based at~$g$ for some~$n\geqslant0$, its adjoint is an~$\Eoo$ map
\begin{equation}\label{eq:induces_derivation}
	E\longrightarrow F^{\mathrm{S}^n}.
\end{equation}
The~$\K$-homology of~$F^{\mathrm{S}^n}$ takes the form
\begin{displaymath}
	\K_*F^{\mathrm{S}^n}\cong \K_*F\oplus\Sigma^{-n}\K_*F
\end{displaymath}
for some~$\uptheta$-module~$\Sigma^{-n}\K_*F$ over~$\K_*F$, a shifted
and twisted copy of~$\K_*F$ itself,
see~\cite{goerss+hopkins:obstructions}, Example~2.2.9, where the
notation~$\Omega$ is used instead of~$\Sigma^{-1}$. The map induced
by~\eqref{eq:induces_derivation} in~$\K$-homology is~$\K_*g$ in the
first factor, and a derivation~\hbox{$\K_*E\rightarrow \K_*F$} in the
second factor. In fact, it is a~$\uptheta$-derivation,
see~\cite{goerss+hopkins:obstructions}, Section~2.4.3. The obstruction
groups are the derived functors of these.

%%%%%%%%%%%%%%%%%%%%%%%%%%%%%%%%%%%%%%%%%%%%%%%%%%%%%%%%%%%%%%%%%%%%%

\subsection{Techniques for computing the obstruction
  groups}\label{subsec:techniques}
  
If~$B_*$ is an even periodic~$\uptheta$-algebra, the (cohomology of
the) cotangent complex~$\L_{B_*/\K_*}$ inherits the structure of
a~$\uptheta$-module over~$B_*$. This is easy to see for the cotangent
module itself: consider the isomorphism between
derivations~\hbox{$B_*\rightarrow M_*$} and algebra
maps~$B_*\rightarrow B_*\oplus M_*$ over~$B_*$, where~$\uptheta$ acts
on the right hand side by
\begin{eqnarray*}
  \uptheta(b,m)
  &=&\uptheta((b,0)+(0,m))\\
  &=&\uptheta(b,0)+\uptheta(0,m)-\frac1p\sum_{j=1}^{p-1}\binom{p}{j}(b,0)^j(0,m)^{p-j}\\
  &=&(\uptheta(b),0)+(0,\uptheta(m))-(b^{p-1},0)(0,m)\\
  &=&(\uptheta(b),\uptheta(m)-b^{p-1}m)
\end{eqnarray*}
Writing~$m=D(b)$, this shows that~$\uptheta$ acts on a derivation~$D$ as
\begin{displaymath}
	(\uptheta D)b=D(\uptheta b)+b^{p-1}Db,
\end{displaymath}
see~\cite{goerss+hopkins:obstructions}, Section~2.4.3. 

This observation allows us to treat the two
problems separately: that of deforming the algebra first and that of
deforming the~$\uptheta$-action later. For practical purposes, this
manifests in a composite functor spectral sequence which takes the
following form.

\begin{proposition}\label{prop:GSS1}
  {\upshape(\cite{goerss+hopkins:obstructions}, (2.4.7))}
  There is a spectral sequence
  \begin{displaymath}
    \Ext^m_{\uptheta\Mod/\K_*}(\H^n(\L_{B_*/\K_*}),M_*)
    \Longrightarrow\D^{m+n}_{\uptheta\Alg/\K_*}(B_*/\K_*,M_*).
  \end{displaymath}
\end{proposition}

In our cases of interest, the algebra~$B_*$ will always be smooth over
$\K_*$, and this implies that the spectral sequence degenerates to
give the isomorphism
\begin{equation}\label{eq:smooth}
    \D^s_{\uptheta\Alg/\K_*}(B_*/\K_*,M_*)\cong
    \Ext^s_{\uptheta\Mod/\K_*}(\Omega_{B_*/\K_*},M_*)
\end{equation}
from~\cite{goerss+hopkins:obstructions} (2.4.9).

In the same vein, the action of the~$p$-adic units through the Adams
operations~$\uppsi^k$ on~$M_*$ can be separated from the action
of~$\uptheta$: As morphisms have to be compatible with both, there is
a Grothendieck spectral sequence as follows.

\begin{proposition}\label{prop:GSS2}
  There is a spectral sequence
\begin{displaymath}
    \Ext^m_{A_*[\uptheta]}(\Omega_{A_*/\K_*},\H^n(\ZZ_p^{\times},M_*))
    \Longrightarrow\Ext^{m+n}_{\uptheta\Mod/\K_*}(\Omega_{A_*/\K_*},M_*)
\end{displaymath}
\end{proposition}

Thus, if the action of the~$p$-adic units through the Adams
operations~$\uppsi^k$ on~$M_*$ is induced, we may also eliminate this
action from the picture to obtain
\begin{equation}\label{eq:induced}
	    \Ext^s_{\uptheta\Mod/\K_*}(\Omega_{A_*/\K_*},M_*)
	    \cong
	    \Ext^s_{A_*[\uptheta]}(\Omega_{A_*/\K_*},M_*^{\ZZ_p^{\times}})
\end{equation}
as in~\cite{goerss+hopkins:obstructions} (2.4.10). The right hand side
turns out to be manageable in the cases relevant here.

%%%%%%%%%%%%%%%%%%%%%%%%%%%%%%%%%%%%%%%%%%%%%%%%%%%%%%%%%%%%%%%%%%%%%

\section{Applications to~$\Eoo$ structures on K3
    spectra}\label{sec:applications}
    
  Let~$(X,L)$ be a~$p$-primitively polarized K3 surface as before. In
  this section we will see that there is a unique~$\Eoo$ structure on
  the K3 spectrum~$E(X,L)$ over the formal completion~$\O(S_L)$ of the
  local ring of~$\Mord$ at~$(X,L)$. It should be emphasized that,
  while the existence of~$E(X,L)$ as ring spectrum up to homotopy is
  known from~\cite{szymik:K3spectra}, this information will not be
  needed here: the spectrum with an~$\Eoo$ structure is shown to exist
  here.

\subsection{A calculation of the obstruction groups}

We would like to have an even periodic~$\Eoo$ ring spectrum~$E(X,L)$ with
\begin{displaymath}
	\pi_{2n}E(X,L)\cong\H^0(S_L,\omega^{\otimes n})
\end{displaymath}
and
\begin{equation}\label{eq:B}
  \K_{2n}E(X,L)\cong\H^0(T_L,\omega^{\otimes n}).
\end{equation}

Let us write~$A_*$ for the even graded ring
with~$\H^0(S_L,\omega^{\otimes n})$ in degree~$2n$, and~$B_*$ for the
even graded ring with~$\H^0(T_L,\omega^{\otimes n})$ in
degree~$2n$. As has been shown in Section~\ref{sec:theta}, the latter
is a~$\uptheta$-algebra with Adams operations over~$\K_*$ and can
therefore serve as an input for the Goerss-Hopkins obstruction
theory. We shall now study the obstruction groups in the range of
interest.

\begin{proposition} 
  \label{prop:obstruction_groups}
  For the graded~$\uptheta$-algebra~$B_*$ as above, the obstruction groups
  \begin{displaymath}
    \D^{s,t}_{\uptheta\Alg/\K_*}(B_*/\K_*,B_*)
  \end{displaymath}
  vanish for~$s\geqslant 2$.
\end{proposition}

\begin{proof} 
  Using the techniques from Subsection 2.4.3
  in~\cite{goerss+hopkins:obstructions}, as recalled here in
  Subsection~\ref{subsec:techniques}, this can be seen as follows.

  First, as~$T_L$ is Galois
  over~$S_L$, and~$S_L$ smooth over~$\ZZ_p$, the cotangent
  complex~$\L_{B_*/\K_*}$ is discrete, equivalent
  to~$\Omega_{B_*/\K_*}$ concentrated in degree~$0$. Therefore, 
  \begin{equation}\label{eq:eins}
    \D^s_{\uptheta\Alg/\K_*}(B_*/\K_*,B_*)
    \cong
    \Ext^s_{\uptheta\Mod/\K_*}(\Omega_{B_*/\K_*},B_*)
  \end{equation}
  as in~\eqref{eq:smooth}. 

  Second, again since~$T_L$ is Galois over~$S_L$, we
  have~$\Omega_{B_*/\K_*}\cong\Omega_{A_*/\K_*}$ by change-of-rings,
  and obtain
  \begin{equation}\label{eq:zwei}
    \Ext^s_{\uptheta\Mod/\K_*}(\Omega_{B_*/\K_*},B_*)
    \cong
    \Ext^s_{\uptheta\Mod/\K_*}(\Omega_{A_*/\K_*},B_*).
  \end{equation}
  
  Third, we may use~\eqref{eq:induced} to get
  \begin{equation}\label{eq:drei}
    \Ext^s_{\uptheta\Mod/\K_*}(\Omega_{A_*/\K_*},B_*)
    \cong
    \Ext^s_{A_*[\uptheta]}(\Omega_{A_*/\K_*},A_*).
       \end{equation}

  Putting~\eqref{eq:eins},~\eqref{eq:zwei}, and~\eqref{eq:drei}
  together yields an isomorphism
  \begin{displaymath}
    \D^s_{\uptheta\Alg/\K_*}(B_*/\K_*,B_*)
    \cong
    \Ext^s_{A_*[\uptheta]}(\Omega_{A_*/\K_*},A_*).
  \end{displaymath}
  
  The~$\Ext$-groups into any module~$M_*$ can be calculated by the resolution
  \begin{displaymath}
    0\longrightarrow
    A_*[\uptheta]\otimes_{A_*}\Omega_{A_*/\K_*}
    \overset{\uptheta}{\longrightarrow}
    A_*[\uptheta]\otimes_{A_*}\Omega_{A_*/\K_*}
    \longrightarrow
    \Omega_{A_*/\K_*}
    \longrightarrow0
  \end{displaymath}
  of~$\Omega_{A_*/\K_*}$.  As~$S_L$ is smooth over~$\ZZ_p$ the
  module~$\Omega_{A_*/\K_*}$ is projective, so that
  \begin{displaymath}
    \Ext_{A_*[\uptheta]}^s(A_*[\uptheta]\otimes_{A_*}\Omega_{A_*/\K_*},M_*)\cong\Ext_{A_*}^s(\Omega_{A_*/\K_*},M_*)=0
  \end{displaymath}
  for all~$s\geqslant1$.  It follows
  that~$\Ext^s_{A_*[\uptheta]}(\Omega_{A_*/\K_*},A_*)$ is zero for
  all~$s\geqslant2$.
\end{proof}

It should be noted that the vanishing of the obstruction groups
for~$s\geqslant2$ implies that the spectral sequences mentioned in the
previous section degenerate at~$E_2$. As we will see, for the even
periodic spectra we will be dealing with, there will neither be
extension problems, so that the homotopy groups of the target can
always be identified with certain obstruction groups in the present
situation.

%%%%%%%%%%%%%%%%%%%%%%%%%%%%%%%%%%%%%%%%%%%%%%%%%%%%%%%%%%%%%%%%%%%%%

\subsection{Existence and uniqueness of~$\Eoo$ structures}

The vanishing of the obstruction groups has the following consequence.

\begin{theorem}\label{thm:Eoo_structure}
	For each~$p$-primitively polarized K3 surface~$(X,L)$ as before,
  there is an even periodic~$\K(1)$-local~$\Eoo$
  ring spectrum~$E(X,L)$ such that
  \begin{eqnarray*}
    \pi_{2n}E(X,L)&\cong&\H^0(S_L,\omega^{\otimes n}),\\
    \K_{2n}E(X,L)&\cong&\H^0(T_L,\omega^{\otimes n}).
  \end{eqnarray*}
  The~$\Eoo$ structure is unique up to equivalence.
\end{theorem}

\begin{proof} 
  By Proposition~\ref{prop:structures},
  the obstructions for existence lie in the groups
  \begin{displaymath}
    \D^{t+2,t}_{\uptheta\Alg/\K_*}(B_*/\K_*,B_*)
  \end{displaymath}
  for~$t\geqslant 1$. These vanish by 
  Proposition~\ref{prop:obstruction_groups}.
  
  Similarly, the obstructions for uniqueness lie in the groups
  \begin{displaymath}
    \D^{t+1,t}_{\uptheta\Alg/\K_*}(B_*/\K_*,B_*)
  \end{displaymath}
  for~$t\geqslant 1$. These vanish by 
  Proposition~\ref{prop:obstruction_groups} as well.
  
  The~$\K$-homology together with the Adams operations determine the homotopy
groups of~$E(X,L)$ in the sense that there is a spectral
sequence
\begin{displaymath}
	\pi_{t-s}E(X,L)\Longleftarrow\H^s(\ZZ_p^\times,B_t)
\end{displaymath}
converging to them. As~$p$ is odd, the cohomological dimension
of~$\ZZ_p^\times$ is~$1$, and the spectral sequence degenerates to the
long exact sequence induced by the~$\K(1)$-local fibration
\begin{displaymath}
	\mathrm{S}^0\longrightarrow\K\stackrel{\uppsi^g-\id}{\longrightarrow}\K,
\end{displaymath}
where~$g$ is a topological generator of~$\ZZ_p^\times$.
As~$B_*$ is concentrated in even degrees, this implies
\begin{displaymath}
	\pi_{2n}E(X,L)\cong\H^0(\ZZ_p^\times,B_{2n})\quad\text{and}\quad
	\pi_{2n-1}E(X,L)\cong\H^1(\ZZ_p^\times,B_{2n}).
\end{displaymath}
And as~$B_*$ is Galois over~$A_*$, this implies 
\begin{displaymath}
  \pi_{2n}E(X,L)\cong (B_{2n})^{\ZZ_p^\times}\cong A_{2n}=\H^0(S_L,\omega^{\otimes n})
\end{displaymath}
as well as~$\pi_{2n-1}E(X,L)=0$.
\end{proof}

Let us now turn our attention to the possible equivalences of~$E(X,L)$.

%%%%%%%%%%%%%%%%%%%%%%%%%%%%%%%%%%%%%%%%%%%%%%%%%%%%%%%%%%%%%%%%%%%%%

\section{Symmetries}\label{sec:symmetries}

Let~$(X,L)$ be an ordinary~$p$-primitively polarized K3 surface as
before. In this section we will investigate how the action of the
automorphism group of~$(X,L)$ on its universal deformation can be
lifted into the world of brave new rings. This is the K3 analogue of
the question settled by Hopkins-Miller in the Lubin-Tate context.

%%%%%%%%%%%%%%%%%%%%%%%%%%%%%%%%%%%%%%%%%%%%%%%%%%%%%%%%%%%%%%%%%%%%%%%%%%%%%

\subsection{Symmetries of K3 surfaces}

Although K3 surfaces have no infinitesimal automorphisms, the group of
automorphisms may nevertheless be infinite. However, the subgroup
preserving a chosen polarization is always finite. This is one reason
to work with polarized K3 surfaces.

A glance into~\cite{mukai} and~\cite{dolgachev+keum} reveals that
there are many simple groups (in the technical sense) which act on K3
surfaces. They cannot be detected by means of the associated formal
Brauer groups, as the finite subgroups of automorphism groups of
formal groups are rather restricted, see~\cite{hewett}.  This is
another argument to tackle K3 surfaces from a crystalline perspective,
due to the following result.

\begin{theorem}{\upshape (\cite{ogus:supersingular}, 2.5, \cite{berthelot+ogus}, 3.23)}\label{thm:ogus}
	If~$p$ is odd, the map
	\begin{displaymath}
		\Aut(X)\longrightarrow\Aut(\Hcris^2(X/W))
	\end{displaymath}
	is injective.
\end{theorem}

While the classification of finite groups of symmetries of {\it
  complex} K3 surfaces has been worked out some time ago,
see~\cite{nikulin} and~\cite{mukai}, the situation in positive
characteristic is more complicated, partially due to the existence of
wild automorphisms: an automorphism of a K3 surface in
characteristic~$p$ is called~{\it wild} if~$p$ divides its
order. Similarly, a group of automorphisms is wild if it contains a
wild automorphism; otherwise it is~{\it tame}.

\begin{theorem}{\upshape (\cite{dolgachev+keum}, 2.1)}\label{thm:dolgachev+keum}
  If~$p>11$, the automorphism group of a K3 surface in posi\-tive
  characteristic~$p$ is tame.
\end{theorem}

The authors also show by means of examples that their bound is
sharp. It seems a remarkable coincidence that this bound is~$1$ larger
than the bound~$h\leqslant10$ for the height of the formal Brauer
group (if finite); the same happens in the case of elliptic curves.

%%%%%%%%%%%%%%%%%%%%%%%%%%%%%%%%%%%%%%%%%%%%%%%%%%%%%%%%%%%%%%%%%%%%%%%%%%%%%

\subsection{Existence and classification of~$\Eoo$ maps}

Let~$(X,L)$ be an ordinary~$p$-primitively polarized K3 surface as
before.  We have already seen, in Theorem~\ref{thm:Eoo_structure},
that there is an~$\Eoo$ structure on the K3 spectrum~$E(X,L)$ which is
unique up to equivalence. While this means that two different models
will be equivalent, there may be many different equivalences between
them. We will now see that automorphisms of~$(X,L)$ give rise
to~$\Eoo$ automorphisms of~$E(X,L)$.

\begin{proposition}\label{prop:Eoo_maps}
	The Hurewicz map
	\begin{displaymath}
	\pi_0\calEoo(E(X,L),E(X,L))
	\longrightarrow
	\operatorname{Hom}_{\uptheta\Alg/\K_*}(\K_*E(X,L),\K_*E(X,L))
	\end{displaymath}
	is bijective.
\end{proposition}

\begin{proof}
	Let us abbreviate~$A_*=\pi_*E(X,L)$ and~$B_*=\K_*E(X,L)$ as before.

  By Proposition~\ref{prop:maps-existence},
  the obstructions to surjectivity of the Hurewicz map lie in the groups
  \begin{displaymath}
    \D^{t+1,t}_{\uptheta\Alg/\K_*}(B_*/\K_*,B_*)
  \end{displaymath}
  for~$t\geqslant 1$. These vanish by 
  Proposition~\ref{prop:obstruction_groups}.
  
  Similarly, by Proposition~\ref{prop:maps-uniqueness}, 
  the obstructions for uniqueness lie in the groups
  \begin{displaymath}
    \D^{t,t}_{\uptheta\Alg/\K_*}(B_*/\K_*,B_*)
  \end{displaymath}
  for~$t\geqslant 1$. These vanish by 
  Proposition~\ref{prop:obstruction_groups} except possibly for 
  the group
  \begin{displaymath}
    \D^{1,1}_{\uptheta\Alg/\K_*}(B_*/\K_*,B_*)
    \cong
    \Ext^1_{A_*[\uptheta]}(\Omega_{A_*/\K_*},\Sigma^{-1}A_*).
  \end{displaymath}
  The last steps in the proof of
  Proposition~\ref{prop:obstruction_groups} have shown that this group
  can be computed as the cokernel of an endomorphism of
	\begin{equation}\label{eq:zero_group}
  	\operatorname{Hom}_{A_*[\uptheta]}(A_*[\uptheta]\otimes_{A_*}\Omega_{A_*/\K_*},\Sigma^{-1}A_*)
	\cong
	\operatorname{Hom}_{A_*}(\Omega_{A_*/\K_*},\Sigma^{-1}A_*).
\end{equation}
However, as~$A_*$ is concentrated in even degrees,~$\Sigma^{-1}A_*$ is
concentrated in odd degrees. It follows that the
group~\eqref{eq:zero_group} itself is already zero.
\end{proof}

As a consequence of the previous proposition, in order to define
homotopy classes of~$\Eoo$ maps on~$E(X,L)$, we merely need to guess
their effect in~$\K$-homology. As with the~$\K$-homology of~$E(X,L)$
itself, the geometry of~$(X,L)$ provide us with the required
information. Here, we may use the fact that the automorphism
group~$\Aut(X,L)$ acts on the universal formal deformation by changing
the identification of the special fiber with~$(X,L)$. As automorphisms
of K3 surfaces are rigid, this can also be understood as follows: an
automorphism~$g$ of~$(X,L)$ sends an~$A$-point~$s$ of~$S_L$ to~$s'$
if~$g$ extends to an isomorphism between the deformations~$(X_s,L_s)$
and~$(X_{s'},L_{s'})$ corresponding to~$s$ and~$s'$. Either way, the
action of~$\Aut(X,L)$ on~$S_L$ can be extended to~$T_L$ as follows. As
the map~\hbox{$T_L\rightarrow S_L$} should be equivariant, we only
need to consider two~$A$-points~$t$ and~$t'$ over~$s$ and~$s'$
where~\hbox{$gs=s'$} as above. Then~$gt=t'$ holds for the action
on~$T_L$ if the extension of~$g$ is compatible with the
chosen~$a$-parts.

\begin{proposition}\label{prop:compatibility}
  The action of~$\Aut(X,L)$ on~$T_L$ respects the structure of
  a~$\uptheta$-algebra with Adams operations on~$\O(T_L)$ defined in
  Section~\ref{sec:theta}, so that there is a factorization
\begin{displaymath}
	\Aut(X,L)\longrightarrow
	\Aut_{\uptheta\Alg/\K_*}(\K_*E(X,L))\stackrel{\subseteq}{\longrightarrow}
	\Aut(\O(T_L))
\end{displaymath}
of this action through the corresponding subgroup.
\end{proposition}

\begin{proof}
  The compatibility of the action of~$\Aut(X,L)$ with the Adams
  operations follows immediately from the fact that~$T_L\rightarrow
  S_L$ is equivariant. For the same reason we may check the
  compatibility with~$\uppsi^p$ (hence~$\uptheta$) on~$S_L$. But for
  every automorphism~$g$ of~$(X,L)$, the
  conjugate~$g\,\uppsican\,g^{-1}$ satisfies the characterization of
  the Katz lift, so that~$g\,\uppsican\,g^{-1}=\uppsican$. This shows
  that~$\uppsican$ is equivariant.
\end{proof}

\parbox{\linewidth}{\begin{theorem}
    For all ordinary~$p$-primitively polarized K3 surfaces~$(X,L)$,
    there is a unique homotopy action 
    \begin{displaymath}
      \Aut(X,L)\longrightarrow\Aut_{\mathrm{Ho}(\calEoo)}(E(X,L))
    \end{displaymath}
    of its automorphism group through~$\Eoo$ maps on the associated~$\Eoo$ ring spectrum.
  \end{theorem}}

\begin{proof}
  The theorem is a consequence of Proposition~\ref{prop:Eoo_maps},
  which implies that the right hand side is isomorphic
  to~$\Aut_{\uptheta\Alg/\K_*}(\K_*E(X,L))$, and
  Proposition~\ref{prop:compatibility}, which provides for the
  required action on~$\K_*E(X,L)$.
\end{proof}

%%%%%%%%%%%%%%%%%%%%%%%%%%%%%%%%%%%%%%%%%%%%%%%%%%%%%%%%%%%%%%%%%%%%%%%%%%%%%

\subsection{Rigidification}

There is another obstruction theory to decide when a homotopy action
of a group can be rigidified to a topological action, see~\cite{cooke}
in the context of topological spaces. This can be used to prove the
following result.

\begin{theorem}\label{thm:tame}
  If the automorphism group of a~$p$-primitively polarized K3
  surface~\hbox{$(X,L)$} is tame, it acts through~$\Eoo$ maps on the
  associated~$\Eoo$ ring spectrum~$E(X,L)$.
\end{theorem}

\begin{proof}
  The obstructions to rigidification lie in the groups
  \begin{displaymath}
    \H^n(\Aut(X,L),\pi_{n-2}\Aut_\id\!E(X,L)\,)
  \end{displaymath}
  for~$n\geqslant 3$, where~$\Aut_\id\!E(X,L)$ is the identity
  component of the derived space of selfequivalences
  of~$E(X,L)$. As~$E(X,L)$ is~$p$-complete, and~$p$ does not divide
  the order of~$\Aut(X,L)$ by assumption, these groups vanish.
\end{proof}

The proof shows that the rigidification is unique up to unique
equivalence under the given hypothesis. Also note that the action is
faithful by definition: it can be detected in~$\K$-homology.

\begin{corollary}\label{cor:11}
  If~$p>11$, the automorphism group of a~$p$-primitively polarized~K3
  surface~\hbox{$(X,L)$} acts through~$\Eoo$ maps on the
  associated~$\Eoo$ ring spectrum~$E(X,L)$.
\end{corollary}	

\begin{proof}
  This follows immediately from Theorem~\ref{thm:dolgachev+keum} and
  the previous result.
\end{proof}

%%%%%%%%%%%%%%%%%%%%%%%%%%%%%%%%%%%%%%%%%%%%%%%%%%%%%%%%%%%%%%%%%%%%%%%%%%%%%

\subsection{Examples: Fermat quartics}

Let us consider the Fermat quartics~$X$ defined
by~\hbox{$T^4_1+T^4_2+T^4_3+T^4_4$} in~$\mathbb{P}^3_k$ with the
polarization~$L=\O(1)$ given by its projective embedding. The Fermat
quartic over a field of odd characteristic~$p$ is known to be ordinary
if and only if~$p\equiv1$ modulo~$4$, see~\cite{artin}
and~\cite{artin+mazur}. The cases~\hbox{$p\geqslant13$} can be dealt
with using Corollary~\ref{cor:11}, and the case~$p=5$ can be dealt
with by hand: Since restriction induces an
isomorphism~$\H^0(\mathbb{P}^3_k,\O(1))\cong\H^0(X,L)$, every
automorphism of~$X$ which preserves the polarization~$L$ extends
uniquely over~$\mathbb{P}^3_k$. The subgroup of~$\mathrm{PGL}_4(k)$
which preserves the Fermat quartic has been determined by Oguiso,
see~\cite{shioda}, in the case~$p\not=3$. This shows that~$\Aut(X,L)$
is an extension of the symmetric group~$\Sigma_4$, which acts by
permutations of the coordinates, by the group of diagonal matrices
with~$4$-th roots of unity as entries. As this group does not contain
an element of order~$5$, Theorem~\ref{thm:tame} can be used to show
that the automorphism group of the Fermat quartic~$(X,L)$ acts
faithfully through~$\Eoo$ maps on the associated~$\Eoo$ ring
spectrum~$E(X,L)$ for all primes~$p$ where the Fermat quartic is
ordinary. The homotopy fixed point spectral sequence
\begin{displaymath}
  \H^s(\Aut(X,L),\H^0(S_L,\omega^{\otimes t/2}))\Longrightarrow \pi_{t-s}E(X,L)^{\mathrm{h}\!\Aut(X,L)}
\end{displaymath}
collapses to give
\begin{displaymath}
	 \pi_t(E(X,L)^{\mathrm{h}\!\Aut(X,L)})
	 \cong
	 \H^0(S_L,\omega^{\otimes t/2})^{\Aut(X,L)}
\end{displaymath}
in this and similar cases. 

%%%%%%%%%%%%%%%%%%%%%%%%%%%%%%%%%%%%%%%%%%%%%%%%%%%%%%%%%%%%%%%%%%%%%

%%%%%%%%%%%%%%%%%%%%%%%%%%%%%%%%%%%%%%%%%%%%%%%%%%%%%%%%%%%%%%%%%%%%%

\parbox{\linewidth}{%
Markus Szymik\newline
Fakult\"at f\"ur Mathematik\newline
Ruhr-Universit\"at Bochum\newline
44780 Bochum\newline
Germany\newline
\phantom{ }
\newline
+49-234-32-23216 (phone)\newline
+49-234-32-14750 (fax)
\newline
\phantom{ }
\newline
{markus.szymik@ruhr-uni-bochum.de}\hfill}


\begin{thebibliography}{MMSS01}

\bibitem[Ati66]{atiyah} M.F. Atiyah. Power operations in
  K-theory. Quart. J. Math. Oxford Ser. 17 (1966) 165--193.
  
\bibitem[Art74]{artin} M. Artin. Supersingular K3 surfaces.
  Ann. Sci. Ecole Norm. Sup. (4) 7 (1974) 543--567.
  
\bibitem[AM77]{artin+mazur} M. Artin, B. Mazur. Formal groups arising
  from algebraic varieties. Ann. Sci. Ecole Norm. Sup. (4) 10 (1977)
  87--131.
  
\bibitem[Beh]{behrens} M. Behrens. Notes on the construction of~{\it
    tmf}. Preprint.

\bibitem[BL]{behrens+lawson} M. Behrens, T. Lawson. Topological
  automorphic forms. Preprint.

\bibitem[BO83]{berthelot+ogus} P. Berthelot, A. Ogus. F-isocrystals
  and de Rham cohomology I.  Invent. Math.  72 (1983) 159--199.

\bibitem[Coo78]{cooke} G. Cooke. Replacing homotopy actions by
  topological actions.  Trans. Amer. Math. Soc.  237 (1978) 391--406.

\bibitem[DI81a]{deligne+illusie:relevements} P. Deligne,
  L. Illusie. Rel\`{e}vement des surfaces K3 en carac\-t\'eris\-tique
  nulle. Algebraic surfaces (Orsay 1976--78) 58--79. Lecture Notes in
  Math. 868. Springer, Berlin-New York, 1981.

\bibitem[DI81b]{deligne+illusie:cristaux} P. Deligne, L. Illusie.
  Cristaux ordinaire et coordonn\'ees canoniques.  Algebraic surfaces
  (Orsay 1976--78) 80--127. Lecture Notes in Math. 868. Springer,
  Berlin-New York, 1981.
  
\bibitem[DK09]{dolgachev+keum} I.V. Dolgachev, J. Keum. Finite groups
  of symplectic automorphisms of K3 surfaces in positive
  characteristic. Ann. of Math. (2) 169 (2009) 269--313.

\bibitem[Dri76]{drinfeld} V.G. Drinfeld.  Coverings of~$p$-adic
  symmetric domains.  Funct. Anal. Appl. 10 (1976) 107--115.
  
\bibitem[Goea]{goerss:Banff} P.G. Goerss. Realizing families of
  Landweber exact homology theories. Preprint.

\bibitem[Goeb]{goerss:Bourbaki} P.G. Goerss. Topological Modular Forms
  [after Hopkins, Miller, and Lurie]. Preprint.

\bibitem[GH00]{goerss+hopkins:andre} P.G. Goerss,
  M.J. Hopkins. Andr\'e-Quillen (co-)homology for sim\-plicial algebras
  over simplicial operads. Une D\'e\-gus\-ta\-tion Topo\-lo\-gique:
  Homotopy Theory in the Swiss Alps, 41--85. Contemp. Math. 265. American
  Mathematical Society, 2000.

\bibitem[GH04]{goerss+hopkins:summary} P.G. Goerss, M.J. Hopkins.
  Moduli spaces of commutative ring spectra.  Structured
  Ring Spectra, 151--200. Cambridge University Press, 2004.

\bibitem[GH]{goerss+hopkins:obstructions} P.G. Goerss,
  M.J. Hopkins. Moduli Problems for Structured Ring Spectra. Preprint.
  
\bibitem[Hew95]{hewett} T. Hewett. Finite subgroups of division
  algebras over local fields.  J. Algebra 173 (1995) 518--548.

\bibitem[Hop95]{hopkins:ICM1994} M.J. Hopkins. Topological modular
  forms, the Witten genus, and the theorem of the cube.  Proceedings
  of the International Congress of Mathematicians, Vol. 1, 2 (Z\"urich
  1994), 554--565, Birkh\"auser, Basel, 1995.

\bibitem[Hop02]{hopkins:ICM2002} M.J. Hopkins.  Algebraic topology and
  modular forms.  Proceedings of the International Congress of
  Mathematicians, Vol. I (Beijing 2002), 291--317. Higher Ed. Press,
  Beijing, 2002.

\bibitem[Hop]{hopkins:K(1)} M.J. Hopkins.~$\K(1)$-local~$\Eoo$
  spectra. Preprint.

\bibitem[HMa]{hopkins+miller:Lubin-Tate} M.J. Hopkins, H.R. Miller.
  Lubin-Tate Deformations in Algebraic Topology.  Preprint.

\bibitem[HMb]{hopkins+miller:elliptic} M.J. Hopkins, H.R. Miller.
  Elliptic curves and stable homotopy I.  Preprint.

\bibitem[Ill79]{illusie:deRham-Witt} L. Illusie. Complexe de de
  Rham-Witt et cohomologie cristalline. Ann. Sci. Ecole Norm. Sup. (4)
  12 (1979) 501--661.
  
\bibitem[Ill85]{illusie:barsotti-tate} L. Illusie. D\'eformations de
  groupes de Barsotti-Tate (d'apr\`es A. Grothendieck). Seminar on
  arithmetic bundles: the Mordell conjecture (Paris,
  1983/84). Ast\'erisque 127 (1985) 151--198.

\bibitem[Kat73]{katz:p-adic properties} N.M. Katz.~$p$-adic properties
  of modular schemes and modular forms. Modular functions of one
  variable III (Antwerp 1972) 69--190. Lecture Notes in Mathematics 350.
  Springer, Berlin, 1973.

\bibitem[Kat75a]{katz:arcata} N.M. Katz.~$p$-adic~$L$-functions via
  moduli of elliptic curves. Algebraic geometry (Arcata 1974) 479--506.
  Amer. Math. Soc., Providence, R.I., 1975.
  
\bibitem[Kat75b]{katz:higher} N.M. Katz. Higher congruences between
  modular forms. Ann. of Math. (2) 101 (1975) 332--367.

\bibitem[Kat77]{katz:eisenstein} N.M. Katz. The Eisenstein measure
  and~$p$-adic interpolation. Amer. J. Math. 99 (1977) 238--311.

\bibitem[Kat81a]{katz:appendix} N.M. Katz. Appendix to Expose
  V. Algebraic surfaces (Orsay 1976--78) 128--137. Lecture Notes in
  Math. 868. Springer, Berlin-New York, 1981.

\bibitem[Kat81b]{katz:serre-tate} N.M. Katz. Serre-Tate local
  moduli. Algebraic surfaces (Orsay 1976--78) 138--202. Lecture
  Notes in Math. 868. Springer, Berlin-New York, 1981.

\bibitem[KM85]{katz+mazur} N.M. Katz, B. Mazur. Arithmetic moduli of
  elliptic curves. Annals of Mathematics Studies, 108. Princeton
  University Press, Princeton, NJ, 1985.

\bibitem[LST64]{lubin+serre+tate} J. Lubin, J.-P. Serre, 
  J. Tate. Elliptic curves and formal groups. Summer Institute on
  Algebraic Geometry, Whitney Estate, Woods Hole, Massachusetts, 1964.

\bibitem[LT66]{lubin+tate} J. Lubin, J. Tate. Formal moduli for
  one-parameter formal Lie groups. Bull. Soc. Math. France 94 (1966)
  49--59.
  
\bibitem[Lur]{lurie:elliptic} J. Lurie. A Survey of Elliptic
  Cohomology. Preprint.

\bibitem[MMSS01]{mandell+may+schwede+shipley} M.A. Mandell, J.P. May,
  S. Schwede, B. Shipley.  Model categories of diagram spectra.
  Proc. London Math. Soc. (3) 82 (2001), 441--512.

\bibitem[MQRT77]{may+quinn+ray+tornehave} J.P. May, F. Quinn, N. Ray,
  J. Tornehave.~$\Eoo$ ring spaces and~$\Eoo$ ring spectra. Lecture
  Notes in Mathematics, 577. Springer-Verlag, Berlin-New York, 1977.

\bibitem[Maz72]{mazur:bulletin} B. Mazur.  Frobenius and the Hodge
  filtration.  Bull. Amer. Math. Soc.  78 (1972) 653--667.

\bibitem[Mes72]{messing:serre-tate} W. Messing. The crystals
  associated to Barsotti-Tate groups: with applications to abelian
  schemes, Lecture Notes in Mathematics 264, Berlin-Heidelberg-New
  York, Springer-Verlag, 1972.
  
\bibitem[Muk88]{mukai} S. Mukai. Finite groups of automorphisms of K3
  surfaces and the Mathieu group.  Invent. Math.  94 (1988) 183--221.
  
\bibitem[Nik80]{nikulin} V.V. Nikulin. Finite groups of automorphisms
  of K\"ahlerian surfaces of type K3. Moscow Math. Soc. 38 (1980)
  71--137.

\bibitem[N83a]{nygaard:Torelli_for_ordinary} N.O. Nygaard. The Torelli
  theorem for ordinary K3 surfaces over finite fields. Arithmetic and
  geometry, Vol. I, 267--276, Progr. Math. 35. Birkh\"auser Boston,
  Boston, MA, 1983.

\bibitem[N83b]{nygaard:Tate_for_ordinary} N.O. Nygaard. The Tate
  conjecture for ordinary K3 surfaces over finite
  fields. Invent. Math. 74 (1983) 213--237.

\bibitem[NO85]{nygaard+ogus:Tate_for_finite} N.O. Nygaard, 
  A. Ogus. Tate's conjecture for K3 surfaces of finite height. Ann. of
  Math. 122 (1985) 461--507.
  
\bibitem[Ogu79]{ogus:supersingular} A. Ogus. Supersingular~K3
  crystals. Journ\'es de G\'eom\'etrie Alg\'e\-brique de Rennes
  (Rennes 1978) II, 3--86. Ast\'erisque 64. Soc. Math. France, Paris, 1979.

\bibitem[Ogu84]{ogus:II} A. Ogus. F-isocrystals and de Rham cohomology
  II. Convergent isocrystals. Duke Math. J. 51 (1984) 765--850.

\bibitem[Rez98]{rezk:hopkins+miller} C. Rezk.  Notes on the
  Hopkins-Miller theorem. Homotopy theory via algebraic geometry and
  group representations (Evanston 1997) 313--366.
  Contemp. Math. 220. Amer. Math. Soc., Providence, 1998.

\bibitem[Riz06]{rizov} J. Rizov. Moduli stacks of polarized K3
  surfaces in mixed characteristic. Serdica Math. J. 32 (2006)
  131--178.

\bibitem[Shi88]{shioda} T. Shioda. Arithmetic and geometry of Fermat
  curves. Algebraic Geometry Seminar (Singapore 1987) 95--102. World
  Scientific, 1988.

\bibitem[Szy]{szymik:K3spectra} M. Szymik.~K3 spectra. Preprint.

\end{thebibliography}
\end{document}